\documentclass[12pt]{amsart}
\usepackage{amssymb}
\usepackage{amsmath, amscd}
\usepackage{amsthm}
\usepackage{comment}
\usepackage[table,xcdraw]{xcolor}
\usepackage[colorlinks=true, linkcolor=blue,urlcolor=blue]{hyperref}

\newtheorem{theorem}{Theorem}[section]

\newtheorem{proposition}[theorem]{Proposition}
\newtheorem{lemma}[theorem]{Lemma}

\newtheorem{corollary}[theorem]{Corollary}
\theoremstyle{definition}

\newtheorem{example}[theorem]{Example}
\newtheorem{definition}[theorem]{Definition}

\newtheorem{problem}[theorem]{Problem}

\date{\today}

\begin{document}

\author[P. Danchev]{Peter Danchev}
\address{Institute of Mathematics and Informatics, Bulgarian Academy of Sciences, 1113 Sofia, Bulgaria}
\email{danchev@math.bas.bg; pvdanchev@yahoo.com}

\author[M. Doostalizadeh]{Mina Doostalizadeh}
\address{Department of Mathematics, Tarbiat Modares University, 14115-111 Tehran Jalal AleAhmad Nasr, Iran}
\email{d\_mina@modares.ac.ir; m.doostalizadeh@gmail.com}

\author[O. Hasanzadeh]{Omid Hasanzadeh}
\address{Department of Mathematics, Tarbiat Modares University, 14115-111 Tehran Jalal AleAhmad Nasr, Iran}
\email{o.hasanzade@modares.ac.ir; hasanzadeomiid@gmail.com}

\title[$n$-$UQ$ rings]{A generalization of $UQ$ rings}

\maketitle

\begin{abstract}
We examine the newly defined class of {\it $n$-$UQ$ rings} described by the condition that $u^n - 1 \in QN(R)$ for every unit $u \in U(R)$, where $QN(R)$ denotes the set of quasi-nilpotent elements (see \cite{Tien}). This class naturally extends the recently defined class of rings in \cite{daoa} and \cite{dam}, as well as expectedly generalizes previously explored concepts such as $UJ$, $UU$ and $UQ$ rings. We conduct here a comprehensive structural analysis of these $n$-$UQ$ rings and study their stability under various ring-theoretic constructions including matrix rings, group rings, trivial extensions and power series rings. As a result, several new characterizations are established, thus revealing relevant connections between $n$-$UQ$ rings and fundamental classes of rings such as reduced, clean, exchange, semi-regular and potent rings, respectively. Moreover, we prove that the classes of $n$-$UJ$ and $n$-$UU$ rings are properly contained in the class of $n$-$UQ$ rings. These achievements not only unify and expand existing theories in this branch, but also provide a robust framework for possible further investigations into the interplay between the unit behavior and quasi-nilpotency in noncommutative ring theory.
\end{abstract}

\section{Introduction and Background}

Through the text, let $R$ be an associative {\bf not} necessarily commutative ring with identity. For such a ring, we denote by $U(R)$ the group of units, by $Nil(R)$ the set of nilpotent elements, by $QN(R)$ the set of quasi-nilpotent elements, by $C(R)$ the set of central elements, and by $Id(R)$ the set of idempotent elements of $R$. The Jacobson radical of $R$ is denoted by $J(R)$. Moreover, the rings of all $n \times n$ matrices and all $n \times n$ upper triangular matrices over $R$ are, respectively, designed by $M_n(R)$ and $T_n(R)$ for any $n\in \mathbb{N}$.

Recall also that a ring $R$ is said to be {\it abelian} provided that each idempotent in $R$ lies in its center, that is, $Id(R)\subset C(R)$.

Some extra classical fundamentals that we need in the sequel are like these: A ring $R$ is called {\it regular} (respectively, {\it unit-regular}) in the sense of von Neumann if, for each $a \in R$, there is $x \in R$ (respectively, $x \in U(R)$) such that $axa = a$. In particular, the ring $R$ is said to be {\it strongly regular} if $a \in a^2 R$ for all $a \in R$. Recall, besides, that a ring $R$ is called an {\it exchange ring} if, for each $a \in R$, there is an idempotent $e = e^2 \in aR$ such that $1 - e \in (1 - a)R$. Moreover, $R$ is termed {\it clean} if every element of $R$ can be written as a sum of a unit and an idempotent (see \cite{5}); in case these two elements commute the ring is named {\it strongly clean}. It is principally known that every clean ring is exchange, though the converse implication does {\it not} hold in general; however, it is valid in the abelian case (see \cite[Proposition 1.8]{5}). In addition, a ring $R$ is said to be {\it semi-regular} if the quotient $R/J(R)$ is regular and all idempotents of $R$ lift modulo $J(R)$. It is well confirmed that all semi-regular rings are exchange, yet the converse fails in all generality (see, for instance, \cite{5}).

Concerning the modern terminology, we can say the following: It is well known that the inclusion $1+J(R) \subseteq U(R)$ holds always. Referring to \cite{Danew} or \cite{KLM}, we define a ring $R$ to be a {\it JU-ring} whenever the equality $U(R)=1+J(R)$ holds. More generally, a ring $R$ is called an {\it $n$-UJ ring} provided $u - u^n \in J(R)$ for each $u \in U(R)$, where $n \geq 2$ is a fixed integer, or equivalently, $u^{n-1} - 1 \in J(R)$ for each $u \in U(R)$ and $n\geq 2$. This concept was initially introduced by Danchev in \cite{Dannew}, and later independently by Ko\c{s}an {\it et al.} in \cite{3}. Note that for $n = 1$ these rings coincide exactly with the aforementioned {\it JU-rings}, often termed also as {\it UJ-rings}. It was incidentally shown there that the properties of being semi-regular, exchange and clean are all equivalent for $2n$-UJ rings.

In a related direction, mimicking \cite{DL}, we recall that $R$ is called a {\it UU-ring} whenever $U(R)=1+Nil(R)$ is true, noticing that the containment $1+Nil(R)\subseteq U(R)$ is fulfilled always. Generally, for a fixed integer $n \ge 2$, a ring $R$ is said to be {\it $n$-UU} if $u^n - 1 \in Nil(R)$ for every $u \in U(R)$. This notion was firstly introduced by Danchev in \cite{4} and was further investigated in greater detail in \cite{dam}. In particular, it was shown in \cite{dam} that a ring $R$ is simultaneously $(n-1)$-UU and strongly $\pi$-regular for any $n>1$ if, and only if, $R$ is strongly $n$-nil-clean. These rings serve as a natural generalization of the aforementioned UU-rings. This class of rings was originally introduced by C\u{a}lug\u{a}reanu in \cite{CUU}, and further studied by Danchev and Lam in \cite{DL}, where it was established that $R$ is strongly nil-clean in the sense of \cite{diesl2013} if, and only if, $R$ is an exchange UU-ring. Moreover, in \cite{karimi}, Zhou {\it et al.} proved that $R$ is a UU ring if, and only if, every unit in $R$ is uniquely nil-clean as defined in \cite{diesl2013}.

An element $a$ of a ring $R$ is called {\it quasi-nilpotent} if $1 - ax$ is invertible for every $x \in R$ satisfying $xa = ax$. It is immediate that both $Nil(R)$ and $J(R)$ are (possibly proper) subsets of $QN(R)$. Note that quasi-nilpotent elements play a significant role in exploring the intricate structure of Banach algebras. Additionally, utilizing quasi-nilpotent, several notable concepts have been developed which include (strongly) J-clean rings (\cite{chen2010}), nil-clean rings and its strong and unique versions (\cite{diesl2013}), generalized Drazin inverses (\cite{koliha1996}), quasi-polar rings (\cite{chen2012}), as well as some other ones.

Further on, in 2024, Danchev {\it et al.} introduced \cite{daoa} the concepts of UQ rings and {\it strongly quasi-nil-clean} rings, which serve as non-trivial extensions of {\it UJ rings} and {\it strongly J-clean rings}, respectively. In fact, a ring $R$ is said to be a {\it UQ ring} if $1 + U(R) = QN(R)$. Furthermore, $R$ is called {\it strongly quasi-nil-clean} if every element of $R$ can be expressed as a sum of an idempotent and a quasi-nilpotent element that commute each another. It was shown there that a ring $R$ is strongly quasi-nil-clean if, and only if, it is both strongly clean and UQ.

As a common extension of the aforementioned classes, the novel concept of {\it $n$-UQ rings} was suggested in \cite{Tien} for an arbitrary integer $n>1$ as follows: A ring $R$ is said to be {\it $n$-UQ} if, for each unit $u \in U(R)$, its $n$-th power $u^n$ can be expressed as a sum of an idempotent and a quasi-nilpotent that commute one another; equivalently, it can be written that $u^n = 1 + q$, where $q \in QN(R)$. This class logically expands all {\it UQ rings} (and, thereby, includes the class of unit uniquely clean rings as proposed in \cite{13}), as well as rings possessing precisely two units. While every {\it $n$-UJ} ring (and, hence, every {\it UJ} ring) belongs to the class of {\it $n$-UQ} rings, the opposed claim does {\it not} hold in general, thus highlighting the genuine novelty of this newly defined class.

The principal aim motivating this work is to provide an in-depth characterization of these $n$-UQ rings by thoroughly examining their intrinsic algebraic structures and their interrelations with the well-examined classes of $n$-UU and $n$-UJ rings. Likewise, we seek to uncover new, previously unobserved properties of {\it $n$-UQ} rings, thereby expanding the theoretical framework and enriching the current understanding of ring theory. Our findings are expected to bridge gaps between these various ring classes and to open avenues for future research, especially in the context of applications where the interplay between idempotent and quasi-nilpotent elements governs structural and functional behaviors.

\section{Basic Properties of $n$-$UQ$ Rings}

We begin our extensive work with our basic notion (compare with \cite{Tien} too).

\begin{definition}\label{2.1}
A ring $R$ is called $n$-$UQ$ if, for each $u\in U(R)$, $u^n-1\in QN(R)$, where $n\geq2$ is a fixed integer.	
\end{definition}

The following constructions are worthy of documenting.

\begin{example}
(i) Every $n$-$UJ$ ring is $n$-$UQ$, because $J(R) \subseteq QN(R)$ is always valid. However, the reciprocal assertion does {\it not} necessarily hold: Indeed, let $R$ be the $\mathbb F_2$-algebra generated by $x, y$ satisfying the only relation $x^2=0$. Then, one verifies that $U(R)=1+\mathbb F_2x+ xRx$, so that $R$ is a $UU$ ring viewing \cite[Example 2.5]{DL} whence it is an $n$-$UU$ ring for every $n\geq1$. However, since $Nil(R)\subseteq QN(R)$, we infer that $R$ is an $n$-$UQ$ ring. But, as $J(R)=(0)$, one concludes that $R$ is {\it not} an $n$-$UJ$ ring for each $n\geq1$.
	
(ii) Every $n$-$UU$ ring is $n$-$UQ$, but, again, the reverse is {\it not} generally valid. For example, the ring $S = \mathbb{F}_2[[x]]$ is $n$-$UQ$ for every $n\geq1$ taking into account Corollary \ref{cor five}(v) below. However, for $1 + x \in U(S)$, we have $1 + x \not\in 1 + Nil(S)$. Therefore, $S$ is not a $UU$ ring and hence obviously not $n$-$UU$ for all $n\geq1$. Nevertheless, some integer $m>1$ depending on $n$ could exist such that $S$ to be $m$-$UU$.

(iii) Let $L = R \times S$, where $R$ is the ring described in (i) and $S$ is the ring described in (ii). Now, bearing in mind the next Proposition~\ref{1.1}(ii), the direct product $L$ is an $n$-$UQ$ ring. However, $L$ manifestly fails to be an $n$-$UJ$ ring as well as an $n$-$UU$ ring.
\end{example}

We say a subring $S$ of a ring $R$ is {\it rationally closed} whenever $U(R) \cap S = U(S)$. The following is pretty obvious, so its verification is voluntarily omitted.

\begin{lemma} \label{good subring}
Let $S$ be a rationally closed subring of $R$. Then, $QN(R) \cap S \subseteq QN(S)$.
\end{lemma}

A series of useful technicalities necessary for our further development of the subject states thus (compare with \cite{Tien} as well).

\begin{proposition}\label{1.1}
\begin{enumerate}
\item
Let $R$ be an $n$-$UQ$ ring, where $n$ is an odd number. Then, $2\in J(R)$.
\item
The direct product \(\prod_{i \in I} R_i\) of rings $R_i$ is $n$-$UQ$ if, and only if, each direct component \(R_i\) is $n$-$UQ$.
\item
Let \(R\) be an $n$-$UQ$. For any unital subring \(S\) of \(R\), if \(S \cap QN(R) \subseteq QN(S)\), then \(S\) is an $n$-$UQ$ ring. In particular, the center of \(R\) is too an $n$-$UQ$ ring.
\item
Let $S$ be a rationally closed subring of $R$. If $R$ is an $n$-$UQ$ ring. Then, $S$ is also an $n$-$UQ$ ring.
\item
Let \(R\) be an $n$-$UQ$ ring and let \(e\) be an idempotent of \(R\). Then, \(eRe\) is an $n$-$UQ$ ring.
\end{enumerate}
\end{proposition}

\begin{proof}
\begin{enumerate}
\item
It is straightforward.
\item
As the equalities \(QN(\prod_{i \in I} R_i) = \prod_{i \in I} QN(R_i)\) and \(U(\prod_{i \in I} R_i) = \prod_{i \in I} U(R_i)\) are fulfilled always, the result follows at once.
\item
Let \(v \in U(S)\) \(\subseteq U(R)\). Since \(R\) is $n$-$UQ$, we have \[ v^n-1 \in QN(R) \cap S \subseteq QN(S).\] So, \(S\) is necessarily an $n$-$UQ$ ring.
\item
It is clear by (iii).
\item
Let $u \in U(eRe)$ with inverse $v$. Therefore, $u + (1 - e) \in U(R)$ with inverse $v + (1 - e)$. By hypothesis, $$(u + (1-e))^n = u^n + (1 - e) = 1 + QN(R).$$ So, it must be that $u^n - e \in QN(R)\cap eRe\subseteq QN(eRe)$ owing to \cite[Lemma 3.5]{chen2012}. Hence, $$u^n \in e + QN(eRe) = 1_{eRe} + QN(eRe),$$ and thus $eRe$ is an $n$-$UQ$ ring, as claimed.
\end{enumerate}
\end{proof}

\begin{proposition}\label{2.10}
For any ring \( R \neq \{0\} \) and any integer \( n \geq 2 \), the ring \( M_n(R) \) is not a $(2k-1)$-$UQ$ ring whenever $k\geq 1$.
\end{proposition}

\begin{proof}
Since \( M_2(R) \) is isomorphic to a corner subring of the ring \( M_n(R) \) for \( n \geq 2 \), it suffices to show that \( M_2(R) \) is not a $(2k-1)$-$UQ$. To that target, consider the matrix
\[
A = \begin{pmatrix} 0 & -1 \\ 1 & 0 \end{pmatrix} \in U(M_2(R)).
\]
Thus, $A^{2k-1}=A$ or $A^{2k-1}=-A$. Now, let $M_2(R)$ is $(2k-1)$-$UQ$. If foremost $A^{2k-1}=A$, then we deduce that
\[
B=A-I=\begin{pmatrix} -1 & -1 \\ 1 & -1 \end{pmatrix}\in QN(M_2(R)).
\]
But, we know that $B$ is a unit leading to contradiction.

If now $A^{2k-1}=-A$, it can be derived that $I\in QN(M_2(R))$ and again this leads to contraposition. So, in either case, \( M_2(R) \) is, indeed, not a $(2k-1)$-$UQ$ ring, as desired.
\end{proof}

Let us now recollect that a set $\{e_{ij} : 1 \le i, j \le n\}$ of non-zero elements of $R$ is said to be a {\it system of $n^2$ matrix units}, provided $e_{ij}e_{st} = \delta_{js}e_{it}$, where $\delta_{jj} = 1$ and $\delta_{js} = 0$ for $j \neq s$. In this case, $e := \sum_{i=1}^{n} e_{ii}$ is an idempotent of $R$ and $eRe \cong M_n(S)$, where $$S = \{r \in eRe : re_{ij} = e_{ij}r~~\textrm{for all}~~ i, j = 1, 2, . . . , n\}.$$
Recall also that a ring $R$ is said to be {\it Dedekind-finite}, provided $ab=1$ amounts to $ba=1$ for any two elements $a,b\in R$. In other words, all one-sided inverse elements in the ring must be two-sided inverses.

\begin{proposition}\label{2.11}
Every $(2k-1)$-$UQ$ ring is Dedekind-finite, provided that $k\geq 1$.
\end{proposition}

\begin{proof}
If we assume the contrary that $R$ is {\it not} a Dedekind-finite ring, then there are elements $a, b \in R$ such that $ab = 1$ but $ba \neq 1$. Assuming $e_{ij} = a^i(1-ba)b^j$ and $e =\sum_{i=1}^{n}e_{ii}$, there is a non-zero ring $S$ such that $eRe \cong M_n(S)$. However, according to Proposition \ref{1.1}(v), the corner $eRe$ is a $(2k-1)$-$UQ$ ring, so $M_n(S)$ has to be a $(2k-1)$-$UQ$ ring too, which contradicts Proposition \ref{2.10}, as expected.
\end{proof}

A ring $R$ is termed {\it reduced} if it contains {\it no} non-zero nilpotent elements, that is, $Nil(R) = (0)$.

\begin{lemma}\label{2.12}
Let \(R\) be a $(2n-1)$-$UQ$ ring for some $n\geq 1$. If \(J(R) = (0)\) and every non-zero right ideal of \(R\) contains a non-zero idempotent, then \(R\) is reduced.
\end{lemma}

\begin{proof}
Suppose on contrary that \(R\) is {\it not} reduced. Then, there its a non-zero element \(a \in R\) such that \(a^2 = 0\). Exploiting \cite[Theorem 2.1]{9}, there exists an idempotent \(e \in RaR\) such that \(eRe \cong M_2(T)\) for some non-trivial ring \(T\). However, Proposition \ref{1.1}(v) guarantees that $eRe$ is a $(2n-1)$-$UQ$ ring whence $M_2(T)$ is a $(2n-1)$-$UQ$ ring as well which is an absurd consulting with Proposition~\ref{2.10}, as suspected.
\end{proof}

\begin{corollary}\label{factor UQ}
Let $R$ be a ring, $I \subseteq J(R)$, and let $\overline{R} = R/I$ be an $n$-$UQ$ ring. Then, $R$ too is an $n$-$UQ$ ring.
\end{corollary}

\begin{proof}
For any $u \in U(R)$, we have $\bar{u} \in U(\overline{R})$. Since $\overline{R}$ is an $n$-$UQ$ ring, we can write ${\bar{u}}^n = \bar{1} + \bar{q}$, where $\bar{q} \in QN(\overline{R})$. Now, \cite[Lemma 4.1(1)]{daoa} applies to get that $q \in QN(R)$. Therefore, $u^n - (1 + q) \in I$. Thus, there is $p \in I$ such that $u = 1 + (p + q)$. But, applying \cite[Lemma 4.1(2)]{daoa}, we detect $p + q \in QN(R)$. Consequently, $R$ is an $n$-$UQ$ ring, as required.
\end{proof}

\begin{lemma}\label{sum two unit}
Let $R$ be an $n$-$UQ$ ring, and let $\overline{R} = R/J(R)$. The following items hold:
\begin{enumerate}
\item	
For any $u, v \in U(R)$, it is true $u^n + v \neq 1$.
\item	
For any $\bar{u}^n, \bar{v} \in U(\overline{R})$, it is true $\bar{u}^n + \bar{v} \neq \bar{1}$.
\item
For any idempotent $e \in R$ and units $u,v \in U(eRe)$, it is true $u^n + v \neq e$.
\item	
For any $n>1$, there does not exist $0 \neq e^2=e \in R$ such that $eRe \cong M_n(S)$ for some ring $S$.
\end{enumerate}
\end{lemma}

\begin{proof}
\begin{enumerate}
\item
Suppose $u^n + v = 1$. Since $R$ is an $n$-$UQ$ ring, there is $q \in QN(R)$ such that $u^n = 1 + q$. Then, $1 = u^n + v = 1 + q + v$ implying $q\in U(R)$, a false.
\item	
Suppose $\overline{u}^n + \overline{v} = \overline{1}$. We may assume $u, v \in U(R)$, and hence $u^n + v - 1 \in J(R)$. On the other hand, since $R$ is an $n$-$UQ$ ring, there is $q \in QN(R)$ such that $u^n = 1 + q$ yielding $q + v \in J(R)$. Thus, $q \in U(R) + J(R) \subseteq U(R)$, again a false.
\item	
We know that $eRe$ is also $n$-$UQ$ and $1_{eRe}=e$. Consequently, the result follows from (i).
\item	
Since $n>1$, it is well known that $M_n(S)$ contains a corner subring isomorphic to the $2 \times 2$ matrix ring.%So, $eRe$ contains a corner ring isomorphic to the $2 \times 2$ matrix ring.
Thus, without loss of generality, we can assume that $n=2$. Consider now the matrix identity in $M_2(S)$
	\[
	I_2 = \begin{pmatrix} 1 & 1 \\ 1 & 0 \end{pmatrix}^2 + \begin{pmatrix} -1 & -1 \\ -1 & 0 \end{pmatrix},
	\]
expressed as a sum of two invertible matrices. This would provide a solution to the equation $u^2 + v = e$ in $eRe \cong M_2(S)$, contradicting part (iii) and thus substantiating our statement.
\end{enumerate}

This finishes the proof.
\end{proof}

\begin{lemma}\label{exe}
Let $R$ be a potent $n$-$UQ$ ring and $\overline{R} = R/J(R)$. Then, the following points are true:
\begin{enumerate}
\item
For any idempotent $\bar{e} \in \overline{R}$ and units $\bar{u}, \bar{v} \in U(\bar{e}\overline{R}\bar{e})$, it must be that $\bar{u}^n + \bar{v} \neq \bar{e}$.
\item	
There are no idempotents $\bar{e} \in \overline{R}$ for which $\bar{e}\overline{R}\bar{e}$ is isomorphic to $M_2(S)$ for any ring $S$.
\end{enumerate}
\end{lemma}

\begin{proof}
\begin{enumerate}
\item
For any idempotent $\bar{e} \in \overline{R}$, we can lift it to an idempotent $e \in R$, because potentness gives that all idempotents can be lifted modulo the Jacobson radical. Also, the corner ring satisfies the isomorphism $$\bar{e}\overline{R}\bar{e} \cong eRe/J(eRe).$$ As $eRe$ inherits the $n$-$UQ$ property of $R$, the result follows immediately from Lemma \ref{sum two unit}(ii).
\item	
Follows directly via a combination of (i) and Lemma \ref{sum two unit}(iv).
\end{enumerate}

This completes the proof.
\end{proof}

Suppose \( R \) is a ring and suppose \( M \) is a bi-module over \( R \). The trivial extension of \( R \) by \( M \), denoted hereafter as \( T(R, M) \), is defined by the equality
\[
T(R, M) = \{ (r, m) \mid r \in R,\, m \in M \},
\]
with addition given component-wise and multiplication given by
\[
(r, m)(s, n) = (rs, rn + ms).
\]
This construction can be identified with the subring
\[
\left\{ \begin{pmatrix} r & m \\ 0 & r \end{pmatrix} \mid r \in R,\ m \in M \right\}
\]
of the {\it formal matrix ring} \( \begin{pmatrix} R & M \\ 0 & R \end{pmatrix} \).

Moreover, when \( M = R \), we have \( T(R, R) \cong R[x]/\langle x^2 \rangle \), that is the ring consisting of {\it dual members} over \( R \).

Additionally, the set of units of the trivial extension \( T(R, M) \) is defined as
\[
U(T(R, M)) = T(U(R), M).
\]

We standardly designate by \( R[[x]] \) the ring of {\it formal power series} over \( R \).

\medskip

We are now positioned to establish the following statements.

\begin{lemma}\cite[Lemma 2.8]{daoa}\label{basic property}
For rings $R,S$, an $(R,S)$-bi-module $N$ and an $R$-bi-module $M$, the following containments are fulfilled:
\begin{enumerate}
\item	
$\{(r,m) \in T(R,M) \mid r \in QN(R), m \in M\} \subseteq QN(T(R,M))$.
\item	
$\left\{\begin{pmatrix} r & m \\ 0 & s \end{pmatrix} \mid r \in QN(R), s \in QN(S), m \in N\right\} \subseteq QN\left(\begin{pmatrix} R & N \\ 0 & S \end{pmatrix}\right)$.
\item	
$\{(a_{ij}) \in T_n(R) \mid a_{ii} \in QN(R) \text{ for } 1 \leq i \leq n\} \subseteq QN(T_n(R))$.
\item	
$\{a_0 + \cdots + a_{n-1}x^{n-1} \in R[x]/\langle x^n \rangle \mid a_0 \in QN(R)\} \subseteq QN(R[x]/\langle x^n \rangle)$.
\item	
$\{a_0 + a_1x + a_2x^2 + \cdots \in R[[x]] \mid a_0 \in QN(R)\} \subseteq QN(R[[x]])$.
\end{enumerate}
\end{lemma}

An immediate consequence asserts the following.

\begin{corollary} \label{cor five}
Let $R$ and $S$ be rings, $N$ an $(R,S)$-bi-module and $M$ an $R$-bi-module. Then, the following issues hold:
\begin{enumerate}
\item	
The trivial extension $T(R,M)$ is $n$-$UQ$ if, and only if, $R$ is $n$-$UQ$.
\item	
The formal triangular matrix ring $\begin{pmatrix} R & N \\ 0 & S \end{pmatrix}$ is $n$-$UQ$ if, and only if, both $R$ and $S$ are $n$-$UQ$.
\item	
For any $n \geq 1$, the triangular matrix ring $T_n(R)$ is $n$-$UQ$ if, and only if, $R$ is $n$-$UQ$.
\item	
For any $n \geq 1$, the quotient ring $R[x]/\langle x^n \rangle$ is $n$-$UQ$ if, and only if, $R$ is $n$-$UQ$.
\item	
The power series ring $R[[x]]$ is $n$-$UQ$ if, and only if, $R$ is $n$-$UQ$.
\end{enumerate}
\end{corollary}

\begin{proof}
We need to establish only case (i), because the remaining cases follow from Lemma \ref{basic property} using a similar line of reasoning.

To that goal, assume that \( T(R, M) \) is an $n$-$UQ$ ring. Since \( R \cong T(R, 0) \), we may interpret \( R \) as a subring of \( T(R, M) \). Hence, Proposition \ref{1.1}(iv) works to find that \( R \) is an $n$-$UQ$ ring.

Conversely, suppose that \( R \) is an $n$-$UQ$ ring, and choose \( (u, m) \in U(T(R, M)) \). Then, \( u \in U(R) \). Since \( R \) is $n$-UQ, we have \( 1 - u^2 \in QN(R) \). Employing Lemma \ref{basic property}, we obtain
\[
(1, 0) - (u, m)^n = (1 - u^n, *) \in T(QN(R), M) \subseteq QN(T(R, M)).
\]
That is why, \( T(R, M) \) is an $n$-$UQ$ ring, as wanted.
\end{proof}

Let \( Nil_{*}(R) \) denote the {\it prime radical} (also known as the {\it lower nil-radical}) of a ring \( R \) that means the intersection of all {\it prime ideals} of \( R \). It is globally known that \( Nil_{*}(R) \) is a nil-ideal of \( R \).

On the same hand, a ring \( R \) is said to be {\it 2-primal} if its prime radical coincides with the set of all nilpotent elements of \( R \) or, in other words, if \( Nil_{*}(R) = Nil(R) \). It is well-established that both reduced rings and commutative rings are non-trivial examples of 2-primal rings.

As usual, we denote by \( R[x] \) the {\it ring of polynomials} over \( R \).

\medskip

In accordance with our preceding result, a ring $R$ is $n$-UQ absolutely when the power series ring $R[[x]]$ is $n$-UQ. This readily leads us to the question: Under what extra circumstances the polynomial ring $R[x]$ also possesses the $n$-UQ property, provided that $R$ does? In what follows, we aim to address this query remembering firstly a well-known affirmation.

\begin{lemma}\cite[Corollary 3.2.]{daoa}\label{cor 2 primal}
Let $R$ be a 2-primal ring. Then, the following equality is true:
\[ J(R[x]) = QN(R[x]) = \text{Nil}(R)[x] = \text{Nil}_*(R)[x] = \text{Nil}_*(R[x]) \]
\end{lemma}

We now menage to prove the following slightly curious assertion.

\begin{lemma}
Let $R$ be a 2-primal ring. Then, the following four conditions are equivalent:
\begin{enumerate}
\item
$R$ is an $n$-$UU$ ring.
\item
$R[x]$ is an $n$-$UQ$ ring.
\item
$R[x]$ is an $n$-$UJ$ ring.
\item
$R[x]$ is an $n$-$UU$ ring.
\end{enumerate}
\end{lemma}

\begin{proof}
Both implications (iii) \( \Rightarrow \) (ii) and (iv) \( \Rightarrow \) (ii) are pretty simple, since we always have that \( \text{Nil}(R), J(R) \subseteq QN(R) \). The equivalence (ii) \( \Leftrightarrow \) (iii) follows immediately from Lemma~\ref{cor 2 primal}.

(i) \( \Rightarrow \) (iv): Write \( u = \sum_{i=0}^n u_i x^i \in U(R[x]) \). Since \( R \) is 2-primal, \cite[Theorem 2.5]{Chenpr} enables us that \( u_0 \in U(R) \) and \( u_i \in \text{Nil}_*(R) \) for all \( 1 \leq i \leq n \).
As \( R \) is $n$-$UU$, we derive \( 1 - u_0^n \in \text{Nil}(R) \). So, in view of the ideal property of \( \text{Nil}_*(R) \), it follows that
\[
1 - u^n \in (1 - u_0^n) + \text{Nil}_*(R)[x]x \subseteq \text{Nil}_*(R)[x] = \text{Nil}_*(R[x]) \subseteq \text{Nil}(R[x]),
\]
as required.

(iv) \( \Rightarrow \) (i): Given \( u \in U(R) \subseteq U(R[x]) \), then \( 1 - u^n \in \text{Nil}(R[x]) \cap R \subseteq \text{Nil}(R) \). Hence, \( R \) is an $n$-$UU$ ring.

(ii) \( \Rightarrow \) (i): Chosen \( u \in U(R) \subseteq U(R[x]) \), then \( 1 - u^n \in QN(R[x]) = \text{Nil}_*(R)[x] \). Thus, \( 1 - u^n \in \text{Nil}_*(R) = \text{Nil}(R) \).

The claim sustained after all.
\end{proof}

\section{Main Results on $n$-$UQ$ Rings}

This section explores how $n$-UQ rings are related to other important classes of rings such as regular rings, semi-potent rings and potent rings. For clarity, recall that a ring $R$ is called {\it semi-potent} if every one-sided ideal outside $J(R)$ contains a non-zero idempotent. Furthermore, a semi-potent ring is said to be {\it potent} when all idempotents can be lifted modulo $J(R)$.

\medskip

We are now in a position to show validity of the following result.

\begin{theorem}\label{semipotent}
Let $R$ be a semi-potent ring. The following are equivalent:
\begin{enumerate}
\item	
The ring $R/J(R)$ is a (2n-1)-UQ ring.
\item	
The ring $R/J(R)$ has the identity $x^{2n} = x$.
\item	
The ring $R$ is a (2n-1)-UJ ring.
\item	
The ring $R/J(R)$ is a (2n-1)-UU ring.
\end{enumerate}
\end{theorem}

\begin{proof}
The implications (iii) $\Rightarrow$ (iv) $\Rightarrow$ (i) are self-evident.

(i) $\Rightarrow$ (ii): First, in virtue of Lemma \ref{2.12}, we know that \( R/J(R) \) is reduced. Now, suppose there is \( a \in R/J(R) \) such that \( a - a^{2n} \neq 0 \) in \( R/J(R) \). Since \( R/J(R) \) is semi-potent, there is an idempotent \( e = e^2 \in R/J(R) \) such that \( e \in (a - a^{2n})R/J(R) \). So, \( e = (a - a^{2n})b \) for some \( b \in R/J(R) \).

As \( R/J(R) \) is abelian, we extract that
\[
e = ea(1 - a^{2n-1})b = e(1 - a^{2n-1})ab.
\]
Thus, Proposition \ref{2.11} reaches us that \( ea,\, e(1 - a^{2n-1}) \in U(eR/J(R)e) \).

In the other vein, we observe that
\[
(ea)^{2n-1} + e(1 - a^{2n-1}) = e. \qquad (*)
\]
Since \( R/J(R) \) is a \( (2n - 1) \)-UQ ring, Proposition \ref{1.1}(v) ensures that \( eR/J(R)e \) is also \( (2n - 1) \)-UQ. Consequently, equation \( (*) \) contradicts Lemma \ref{sum two unit}. Hence, \( R/J(R) \) satisfies the identity \( x^{2n} = x \).

(ii) $\Rightarrow$ (iii): Put \( u \in U(R) \). Then, by (ii), we arrive at \( u - u^{2n} \in J(R) \). Since \( J(R) \) is an ideal and \( u \) is a unit, it follows that \( 1 - u^{2n-1} \in J(R) \). Thus, \( R \) is a \( (2n - 1) \)-UJ ring, as desired.
\end{proof}

Some additional background is now helpful to be recollected: A ring $R$ is called {\it $\pi$-regular} if, for each $a\in R$, $a^n\in a^nRa^n$ for some integer $n$ depending on $a$. So, regular rings are always $\pi$-regular. Likewise, a ring $R$ is said to be {\it strongly $\pi$-regular} provided that, for any $a\in R$, there exists $n$ depending on $a$ such that $a^n\in a^{n+1}R$.

\medskip

We are now ready to formulate and prove the following consequence.

\begin{corollary}
The following four statements are equivalent for a $(2n-1)$-$UQ$ ring \(R\):
\begin{enumerate}
\item	
$R$ is a regular ring.
\item	
$R$ is a $\pi$-regular reduced ring.
\item	
$R$ is a strongly regular ring.
\item	
$R$ is a unit-regular ring.
\item	
$R$ satisfies the identity $x^{2n}=x$.
\end{enumerate}
\end{corollary}

\begin{proof}
(i) $\Rightarrow$ (ii): Since \( R \) is a regular ring, it must be that \( J(R) = (0) \). Moreover, \( R \) is a semi-potent ring, because, for every \( 0 \neq a \in R \), there is \( b \in R \) such that \( a = aba \). Apparently, \( e := ab \in aR \) is an idempotent. So, Lemma \ref{2.12} assures that \( R \) is reduced.
	
(ii) $\Rightarrow$ (i): This follows from \cite[Theorem 3]{10}.
	
(i) $\Rightarrow$ (iii): Looking at (i) $\Rightarrow$ (ii), we may assume $R$ is reduced (and hence abelian). Therefore, $R$ is a strongly regular ring.
	
(iii) $\Rightarrow$ (iv): This is always true.
	
(iv), (v) $\Rightarrow$ (i): This is quite elementary.
	
(i) $\Rightarrow$ (v): Since every regular ring is semi-potent and $J(R)=(0)$, the result follows at once from Theorem~\ref{semipotent}.
\end{proof}

Our next chief result sounds surprisingly like this.

\begin{theorem}\label{potent}
Let $R$ be a potent ring. Then, the following six conditions are equivalent:
\begin{enumerate}
\item	
$R$ is a $(2n-1)$-$UQ$ ring.
\item	
$R/J(R)$ is a $(2n-1)$-$UQ$ ring.
\item	
$R/J(R)$ satisfies the identity $x^{2n}=x$.
\item	
$R$ is a $(2n-1)$-$UJ$ ring.
\item	
$R/J(R)$ is a $(2n-1)$-$UJ$ ring.
\item	
$R/J(R)$ is a $(2n-1)$-$UU$ ring.
\end{enumerate}
\end{theorem}

\begin{proof}
One observes that Theorem \ref{semipotent} establishes the equivalences (ii) $\Leftrightarrow$ (iii) $\Leftrightarrow$ (iv) $\Leftrightarrow$ (vi), whereas \cite[Proposition 2.9]{3} proves (iv) $\Leftrightarrow$ (v). The implication (iv) $\Rightarrow$ (i) is immediate, leaving only (i) $\Rightarrow$ (iii) to be verified below.

(i) $\Rightarrow$ (iii): Since $R$ is a semi-potent ring, we find that $\overline{R} = R/J(R)$ is also semi-potent. We will show now that $\overline{R}$ is a reduced ring.

To that purpose, suppose $x^2 = 0$ but $0 \neq x \in \overline{R}$. Then, \cite[Theorem 2.1]{9} teaches us that there is $\bar{e} \in \overline{R}$ such that $\bar{e}\overline{R}\bar{e} \cong M_2(S)$, where $S$ is a non-zero ring. This, however. contradicts Lemma \ref{exe}(ii). Hence, $\overline{R}$ is reduced.

Now, suppose there is $a \in \overline{R}$ with $a - a^{2n} \neq 0$ in $\overline{R}$. Since $\overline{R}$ is semi-potent, there is $0\neq e = e^2 \in \overline{R}$ such that $e \in (a - a^{2n})\overline{R}$. Thus, $e = (a - a^{2n})b$ for some $b \in \overline{R}$. Since $e$ is central, we have $$e = ea \cdot e(1-a^{2n-1}) \cdot ey,$$ and so both $ea, e(1-a^{2n-1}) \in U(e\overline{R}e)$. So, $(ea)^{2n-1} + e(1-a^{2n-1}) = e$, contradicting Lemma \ref{exe}(i). Finally, $\overline{R}$ has the identity $x^{2n}=x$, as pursued.
\end{proof}

Three more consequences state as follows.

\begin{corollary}\label{2.16}
Let \( R \) be a $(2n-1)$-$UQ$ ring. Then, the following three points are equivalent:
\begin{enumerate}
\item
\( R \) is an exchange ring.
\item
\( R \) is a clean ring.
\item
\( R \) is a semi-regular ring.
\end{enumerate}
\end{corollary}

\begin{proof}
\((ii) \Rightarrow (i)\): It is clear.

\((i) \Rightarrow (ii)\): If \( R \) is exchange $(2n-1)$-$UQ$, then \( R \) is reduced with Lemma \ref{2.12} at hand, and hence it is abelian. Therefore, \( R \) is abelian exchange, so it is clean (see \cite{5}).

\((i) \Rightarrow (iii)\): Since $R$ is exchange, all idempotents lift modulo \(J(R)\) (see \cite{5}). Also, with the aid of Theorem \ref{potent}, \(R/J(R)\) possesses the identity $x^{2n}=x$. So, $R$ is semi-regular.

\((iii) \Rightarrow (i)\): It is plain, because semi-regular rings are always exchange.
\end{proof}

To simplify exposition, we henceforth will say that a potent ring whose elements modulo Jacobson radical satisfy the equation $x^{2n}=x$ is just {\it semi-$(2n)$-potent}.

\begin{corollary}
Let $R$ be a ring. Then, $R$ is a $(2n-1)$-$UQ$ clean ring if, and only if, $R$ is semi-$(2n)$-potent.
\end{corollary}

\begin{proof}
Adapting Theorem \ref{potent}, every clean $(2n-1)$-UQ ring is semi-$(2n)$-potent.

Conversely, assume $R$ is semi-$(2n)$-potent. Set $u \in U(R)$, and let $u = e + j$ be a semi-$(2n)$-potent representation of $u$ with $e^{2n-1}\in Id(R)$ and $j\in J(R)$. Then, one sees that
\[
	e = u - j \in U(R) + J(R) \subseteq U(R) \Rightarrow e^{2n-1} \in U(R) \cap Id(R) = \{1\}
\]
and, consequently, $u^{2n-1} = e^{2n-1} + j$ for some $j\in J(R)$. Hence, one obtains that $$u^{2n-1}\in 1+J(R)\subseteq 1+QN(R).$$ We now show that $R$ is a clean ring. In fact, since $R/J(R)$ has the identity $x^{2n}=x$, it follows that $R/J(R)$ is strongly regular and so \cite[Theorem 1.1.2]{chenbook} employs to get that it is clean. Moreover, by the definition of semi-$(2n)$-potent rings, idempotents lift modulo $J(R)$. Therefore, $R$ is a clean ring, as asked.
\end{proof}

\begin{corollary}
Let $R$ be an artinian (in particular, a finite) ring. Then, the following three conditions are equivalent:
\begin{enumerate}
\item	
$R$ is a $(2n-1)$-$UQ$ ring.
\item	
$R$ is a $(2n-1)$-$UJ$ ring.
\item	
$R$ is a $(2n-1)$-$UU$ ring.
\end{enumerate}
\end{corollary}

\begin{proof}
%Our investigation is confined to artinian rings, a natural choice given that all finite rings fall within this broader class.
As established in \cite[Corollary 6]{camilo}, every artinian ring is clean which, in conjunction with Theorem \ref{potent}, insures the equivalence between the $(2n-1)$-$UQ$ and $(2n-1)$-$UJ$ conditions in this context.

Moreover, the artinian assumption is a guarantor that the Jacobson radical $J(R)$ is contained in the nil-radical $Nil(R)$. So, \cite[Lemma 2.1]{10} can be applied to get that the ring $R$ satisfies the $(2n-1)$-$UU$ property if, and only if, the factor-ring $R/J(R)$ does. Furthermore, Theorem \ref{potent} tells us that the $(2n-1)$-$UJ$ and $(2n-1)$-$UU$ properties are equivalent for artinian rings, thereby completing the argument.
\end{proof}

We are now concerned with group rings. To that goal, let $R$ be any ring and $G$ any group. We traditionally denote by $RG$ the {\it group ring} of $G$ over $R$. Likewise, the {\it augmentation ideal} $\Delta(RG)$ is defined as the kernel of the standard augmentation homomorphism $\varepsilon: RG\to R$, where $$\varepsilon\left(\sum_{g\in G}a_g g\right) = \sum_{g\in G}a_g.$$

Remember also that a group $G$ is called a {\it $p$-group} if all its elements have orders that are powers of a prime $p$, and {\it locally finite} if every finitely generated subgroup is finite.

\medskip

What we can perceive in this matter are the following two claims.

\begin{proposition}
\begin{enumerate}
\item
If $RG$ is an $n$-$UQ$ ring, then $R$ is also an $n$-$UQ$ ring.
\item	
If $R$ is an $n$-$UQ$ ring and $G$ is a locally finite $p$-group, where $p$ is a prime number such that $p \in J(R)$, then $RG$ is an $n$-$UQ$ ring.
\end{enumerate}
\end{proposition}

\begin{proof}
\begin{enumerate}
\item	
Since $R$ is a rationally closed subring of $RG$, Proposition \ref{1.1}(iv) informs us that $R$ inherits the $n$-$UQ$ property of $RG$.
\item	
Regarding \cite[Lemma 2]{26}, we have $\Delta(RG) \subseteq J(RG)$. So, the isomorphism $RG/\Delta(RG) \cong R$ along with Corollary \ref{factor UQ} give us the conclusion.
\end{enumerate}
\end{proof}

We terminate our work with the following observation.

\begin{lemma}
Let $RG$ be an n-$UQ$ ring with $3 \in J(RG)$, and let $G$ be a $2$-group. Then, $G$ is a group of exponent $2$.
\end{lemma}

\begin{proof}
First of all, we prove for every $g \in G$ and $k \in \mathbb{N}$ that $1+g^{2^k} \in U(RG)$. Since $R\langle g\rangle$ is a rationally closed subring of $RG$, invoking Proposition \ref{1.1}(iv) we deduce that $R\langle g\rangle$ is also an $n$-$UQ$ ring. Therefore, without loss of generality, we may assume that the element $g$ is central.

As $G$ is a $2$-group, let $k \in \mathbb{N}$ be the smallest integer with $g^{2^k}=1$. Consequently, one inspects that $1-g^{2^k}=0$. Since $3 \in J(RG)$, we check that
\[
1+g^{2^k} = 1-g^{2^k} - g^{2^k} + 3g^{2^k} = -g^{2^k} + 3g^{2^k} \in U(RG) + J(RG) \subseteq U(RG).
\]

Now, consider the equality $(1-g^{2^{k-1}})(1+g^{2^{k-1}})=0$. Since $1+g^{2^{k-1}} \in U(RG)$, we must have $1=g^{2^{k-1}}$, which leads to an impossibility. That is why, $G$ must be a group of exponent $2$, as formulated.
\end{proof}

\section{Open Problems}

In closing, we pose two challenging questions hoping to stimulate a further intensive research on the present topic.

\begin{problem}
Let $R$ be a ring and $G$ a multiplicative group. Determine satisfactory necessary and sufficient conditions depending on both $R$ and $G$ under which the group ring $RG$ is $n$-$UQ$ for any $n\geq 1$.
\end{problem}

\begin{problem}
Let $R$ be a $n$-$UQ$ ring with $n\geq2$. Under what suitable conditions on $R$ does the $n$-$UQ$ property imply that $R$ is a $UQ$ ring?
\end{problem}

\end{document}